\documentclass[11pt]{article}

\usepackage{amssymb}
\usepackage{amsmath}
\usepackage{color}
\usepackage{theorem}
\usepackage[all]{xy}

\newcommand{\XYMATRIX}{\xymatrix@M=6pt}
\newcommand{\aremb}{\ar@{^{(}->}}
\newcommand{\arembfrom}{\ar@{<-^{)}}}

\numberwithin{equation}{section}

\theorembodyfont{\slshape}

  \newtheorem{THM}{Theorem}[section]
  \newtheorem{LEM}[THM]{Lemma}

\theorembodyfont{\rmfamily}

  \newtheorem{DEF}[THM]{Definition}
  \newtheorem{EX}{Example}[section]

\newif\ifQEDsign
\newcommand{\QED}{\global\QEDsigntrue\hfill$\square$}

\newenvironment{proof}%
    {\par\noindent\textit{Proof.}\global\QEDsignfalse}%
    {\ifQEDsign\else\QED\fi\par\bigskip\par}

\newcommand{\quotient}[2]{\genfrac{[}{]}{0pt}{}{#1}{#2}}

\newcommand{\alex}{\mathrel{<_{\mathit{alex}}}}

\newcommand{\sqsal}{\mathrel{\sqsubset_{\mathit{sal}}}}

\newcommand{\ssal}{\mathrel{\subset_{\mathit{sal}}}}
\newcommand{\lsal}{\mathrel{<_{\mathit{sal}}}}

\renewcommand{\le}{\leqslant}
\renewcommand{\ge}{\geqslant}

\newcommand{\0}{\varnothing}

\renewcommand{\phi}{\varphi}
\renewcommand{\epsilon}{\varepsilon}
\renewcommand{\rho}{\varrho}

\newcommand{\CC}{\mathbf{C}}
\newcommand{\DD}{\mathbf{D}}

\newcommand{\KK}{\mathbf{K}}

\newcommand{\union}{\cup}

\newcommand{\Boxed}[1]{\mbox{$#1$}}

\newcommand{\id}{\mathrm{id}}

\newcommand{\Ob}{\mathrm{Ob}}

\newcommand{\op}{\mathrm{op}}

\newcommand{\calA}{\mathcal{A}}
\newcommand{\calB}{\mathcal{B}}
\newcommand{\calC}{\mathcal{C}}
\newcommand{\calD}{\mathcal{D}}

\newcommand{\calP}{\mathcal{P}}

\newcommand{\calS}{\mathcal{S}}

\newcommand{\calV}{\mathcal{V}}
\newcommand{\calW}{\mathcal{W}}
\newcommand{\calX}{\mathcal{X}}

\newcommand{\CH}{\mathbf{Ch}_{\mathit{emb}}}
\newcommand{\CHrs}{\mathbf{Ch}_{\mathit{rs}}}

\newcommand{\OOGRAsrq}{\mathbf{OOgra}_{\mathit{srq}}}
\newcommand{\EDIGsrq}{\mathbf{EDig}_{\mathit{srq}}}

\title{A Dual Ramsey Theorem for Finite Ordered Oriented Graphs}
\author{%
  Dragan Ma\v sulovi\'c, Bojana Panti\'c\\
  University of Novi Sad, Faculty of Sciences\\
  Department of Mathematics and Informatics\\
  Trg Dositeja Obradovi\'ca 3, 21000 Novi Sad, Serbia\\
  e-mail: $\{$masul, bojana$\}$@dmi.uns.ac.rs}

\begin{document}
\maketitle

\begin{abstract}
  In contrast to the abundance of ``direct'' Ramsey results for classes of finite structures
  (such as finite ordered graphs, finite ordered metric spaces and finite posets
  with a linear extension), in only a handful of cases we have a meaningful
  dual Ramsey result. In this paper we prove a dual Ramsey theorem for finite ordered oriented graphs.
  Instead of embeddings, which are crucial for ``direct'' Ramsey results, we consider
  a special class of surjective homomorphisms between finite ordered oriented graphs.
  Since the setting we are interested in involves both structures and morphisms, all our results are spelled out using
  the reinterpretation of the (dual) Ramsey property in the language of category theory.

  \bigskip


  \noindent \textbf{AMS Subj.\ Classification (2010):} 05C55, 18A99
\end{abstract}

\section{Introduction}

Generalizing the classical results of F.~P.~Ramsey from the late 1920's, the structural Ramsey theory originated at
the beginning of 1970’s in a series of papers (see \cite{N1995} for references).
We say that a class $\KK$ of finite structures has the \emph{Ramsey property} if the following holds:
for any number $k \ge 2$ of colors and all $\calA, \calB \in \KK$ such that $\calA$ embeds into $\calB$
there is a $\calC \in \KK$
such that no matter how we color the copies of $\calA$ in $\calC$ with $k$ colors, there is a \emph{monochromatic} copy
$\calB'$ of $\calB$ in $\calC$ (that is, all the copies of $\calA$ that fall within $\calB'$ are colored by the same color).
In this parlance the Finite Ramsey Theorem takes the following form:

\begin{THM} (Finite Ramsey Theorem~\cite{Ramsey})\label{dthms.thm.frt}
  The class of all finite chains has the Ramsey property.
\end{THM}

In~\cite{GR} Graham and Rothschild proved their famous Graham-Roth\-schild Theorem,
a powerful combinatorial statement about words intended for dealing with the Ramsey property of certain
geometric configurations. The fact that it also implies the following dual Ramsey statement was recognized
almost a decade later.

\begin{THM} (Finite Dual Ramsey Theorem~\cite{GR,Nesetril-Rodl-DRT})\label{drpog.thm.FDRT}
  For all positive integers $k$, $a$, $m$ there is a positive integer $n$ such that
  for every $n$-element set $C$ and every $k$-coloring of the set $\quotient Ca$ of all partitions of
  $C$ with exactly $a$ blocks there is a partition $\beta$ of $C$ with exactly $m$ blocks such that
  the set of all partitions from $\quotient Ca$ which are coarser than $\beta$ is monochromatic.
\end{THM}

One of the cornerstones of the structural Ramsey theory is the Ne\v set\v ril-R\"odl Theorem which states that the class
of all finite linearly ordered relational structures (all having the same, fixed, relational type) has the Ramsey property
\cite{AH}, \cite{Nesetril-Rodl,Nesetril-Rodl-1983}. The fact that this result has been proved independently by several research teams, and then reproved
in various ways and in various contexts \cite{AH,Nesetril-Rodl-1983,Nesetril-Rodl-1989}
clearly demonstrates the importance and justifies the distinguished status
this result has in discrete mathematics. The search for a dual version of the Ne\v set\v ril-R\"odl Theorem
was and still is an important research direction and several versions of the dual of the
Ne\v set\v ril-R\"odl Theorem have been published, most notably by Pr\"omel in~\cite{Promel-1985},
Frankl, Graham, R\"odl in~\cite{frankl-graham-rodl} and recently by Solecki in~\cite{Solecki-2010}.
In~\cite{masul-DRTRS} we prove yet another dual version of the Ne\v set\v ril-R\"odl Theorem
and, in connection to that, the dual Ramsey statements for finite oriented graphs and hypergraphs.
As a spin-off, we also proved in~\cite{masul-DRTRS} that no reasonable category of finite linearly ordered
tournaments has the dual Ramsey property. This immediately raised the question of a dual Ramsey statement for
finite ordered oriented graphs, which we solve in the present paper. It is important to note that the main result
of this paper can also be derived from the main result of~\cite{Solecki-2010}, but in this paper our goal is to
demonstrate a direct proof.

In its original form, the Ramsey theorem is a statement about coloring $k$-element subsets of $\omega = \{0, 1, 2, \ldots\}$.
A dual statement about coloring $k$-element partitions of $\omega$ was proved in~\cite{Carlson-Simpson}.
These results actually marked the beginning of a search for ``dual'' Ramsey statements,
where instead of coloring substructures we are interested in coloring ``quotients'' of structures.

Going back to the Finite Dual Ramsey Theorem,
it was observed in~\cite{promel-voigt-surj-sets} that each partition of a finite linearly ordered set can be
uniquely represented by the rigid surjection which takes each element of the underlying set to the minimum
of the block it belongs to (see Subsection~\ref{drpog.subset.lo} for the definition of a rigid surjection).
Hence, Finite Dual Ramsey Theorem is a structural Ramsey result about finite chains and special surjections between them.
This result was later generalized to trees in~\cite{solecki-dual-ramsey-trees}, and, using a different set of
techniques, to finite permutations in~\cite{masul-drp-perm}.

In contrast to the on-going Ramsey classification projects
(see for example~\cite{MANY-MANY}) where the research is focused on fine-tuning the objects, in~\cite{masul-DRTRS}
we advocate the idea that fine-tuning the morphisms is the key to proving dual Ramsey results.
Since the setting we are interested in involves
both structures and morphisms, all our results are spelled out using
the categorical reinterpretation of the Ramsey property as proposed in~\cite{masulovic-ramsey}.
Actually, it was Leeb who pointed out already in 1970 that the use of category theory can be quite helpful
both in the formulation and in the proofs of results pertaining to structural Ramsey theory~\cite{leeb-cat}.
In~\cite{masul-DRTRS}, but also in the present paper, we argue that this is even more the case when dealing with
the dual Ramsey property.

In Section~\ref{drpog.sec.prelim} we give a brief overview of certain technical notions referring to
linear orders and oriented graphs.

In Section~\ref{drpog.sec.rplct} we provide basics of category theory and give
a categorical reinterpretation of the Ramsey property
as proposed in~\cite{masulovic-ramsey}. We define the Ramsey property and the dual Ramsey property for
a category and illustrate these notions using some well-known examples.

Finally, in Section~\ref{drpog.sec.main} we prove a
dual Ramsey theorem for finite ordered oriented graphs, which is the main result of the paper.

\section{Preliminaries}
\label{drpog.sec.prelim}

In order to fix notation and terminology in this section we give a brief overview of certain notions referring to
linear orders and oriented graphs.

\subsection{Linear orders}
\label{drpog.subset.lo}

A \emph{chain} is a pair $(A, \Boxed<)$ where $<$ is a linear order on~$A$.
In case $A$ is finite we shall simply write
$\{a_1 < a_2 < \ldots < a_n\}$ instead of $(A, \Boxed<)$.

Let $(A, \Boxed<)$ and $(B, \Boxed\sqsubset)$ be chains such that $A \cap B = \0$.
Then $(A \union B, \Boxed{\Boxed< \oplus \Boxed\sqsubset})$ denotes the \emph{concatenation} of
$(A, \Boxed<)$ and $(B, \Boxed\sqsubset)$, which is a chain on $A \union B$ such that every element of $A$
is smaller then every element of $B$, the elements in $A$ are ordered linearly by~$<$, and the elements of $B$
are ordered linearly by~$\sqsubset$.

Following \cite{promel-voigt-surj-sets} we say that a surjection
$
  f : \{a_1 < a_2 < \ldots < a_n\} \to \{b_1 < b_2 < \ldots < b_k\}
$
between two finite chains is \emph{rigid} if $\min f^{-1}(x) < \min f^{-1}(y)$ whenever $x < y$.
Equivalently, a rigid surjection maps each initial segment of $\{a_1 < a_2 < \ldots < a_n\}$
onto an initial segment of $\{b_1 < b_2 < \ldots < b_k\}$;
other than that, a rigid surjection is not required to respect the linear orders in question.

Every finite chain $(A, \Boxed<)$ induces the \emph{anti-lexicographic order on $A^2$}
as follows: $(a_1, a_2) \alex (b_1, b_2)$ if and only if $a_2 < b_2$, or $a_2 = b_2$ and $a_1 < b_1$.
It also induces the \emph{anti-lexicographic order on $\calP(A)$} as follows.
For $X \in \calP(A)$ let $\vec X \in \{0, 1\}^{|A|}$ denote the characteristic vector of $X$.
(As $A$ is linearly ordered, we can assign a string of 0's and 1's to each subset of $A$.)
Then for $X, Y \in \calP(A)$ we let $X \alex Y$ if and only if $\vec X \alex \vec Y$,
where the vectors are compared with respect to the usual ordering $0 < 1$.
It is easy to see that for $X, Y \in \calP(A)$ we have that $X \alex Y$ if and only if
$X \subset Y$, or $\max(X \setminus Y) < \max(Y \setminus X)$ in case
$X$ and $Y$ are incomparable.

Finally, for a finite chain $(A, \Boxed<)$ let us define the linear order $\lsal$ on $A^2$ as follows
(``sal'' in the subscript stands for ``special anti-lexicographic''; cf. \ the definition of $\lsal$
in~\cite{masul-DRTRS}). Take any $(a_1, a_2), (b_1, b_2) \in A^2$.
\begin{itemize}
\item
  If $a_1 = a_2$ and $b_1 = b_2$ then $(a_1, a_2) \lsal (b_1, b_2)$ if and only if $a_1 < b_1$;
\item
  if $a_1 = a_2$ and $b_1 \ne b_2$ then $(a_1, a_2) \lsal (b_1, b_2)$;
\item
  if $a_1 \ne a_2$, $b_1 \ne b_2$ and $\{a_1, a_2\} = \{b_1, b_2\}$
  then $(a_1, a_2) \lsal (b_1, b_2)$ if and only if
  $(a_1, a_2) \alex (b_1, b_2)$;
\item
  finally, if $a_1 \ne a_2$, $b_1 \ne b_2$ and $\{a_1, a_2\} \ne \{b_1, b_2\}$
  then $(a_1, a_2) \lsal (b_1, b_2)$ if and only if
  $\{a_1, a_2\} \alex \{b_1, b_2\}$.
\end{itemize}

\subsection{Oriented graphs}

An \emph{oriented graph} $\calV = (V, \rho)$ is a set $V$ together with a reflexive binary relation $\rho$
on $V$ such that $(v_1, v_2) \in \rho \Rightarrow (v_2, v_1) \notin \rho$ whenever $v_1 \ne v_2$.
(Note that all the graph-like structures in this paper will be reflexive because our principal structure maps
will be special surjective homomorphisms.) Let $\Delta_V = \{(v, v) : v \in V\}$.

Let $\calV = (V, \rho)$ and $\calW = (W, \sigma)$ be oriented graphs. A mapping $f : V \to W$
is a \emph{homomorphism} from $\calV$ to $\calW$, and we write $f : \calV \to \calW$, if
$(v_1, v_2) \in \rho \Rightarrow (f(v_1), f(v_2)) \in \sigma$ for all $v_1, v_2 \in V$.
A homomorphism $f : \calV \to \calW$ is an \emph{embedding} if $f$ is injective and
$(f(v_1), f(v_2)) \in \sigma \Rightarrow (v_1, v_2) \in \rho$ for all $v_1, v_2 \in V$.
A homomorphism $f : \calV \to \calW$ is a
\emph{quotient map} if $f$ is surjective and for every $(w_1, w_2) \in \sigma$
there exists a pair $(v_1, v_2) \in \rho$ such that $f(v_1) = w_1$ and $f(v_2) = w_2$.

An \emph{ordered oriented graph} is a structure $\calV = (V, \rho, \Boxed{<})$
where $(V, \rho)$ is an oriented graph and $<$ is a linear order on~$V$.

A \emph{digraph with a linear extension} $\calV = (V, \rho, \Boxed{<})$ is a set $V$ together with a
reflexive binary relation $\rho$ and a linear order $<$ on $V$ such that
$(v_1, v_2) \in \rho \Rightarrow v_1 < v_2$ whenever $v_1 \ne v_2$.
Note that such a digraph is necessarily acyclic.

It has been amply demonstrated in~\cite{masul-DRTRS} that the key to providing a structural dual Ramsey result is
the right choice of morphisms. In case of digraphs with linear extensions the following notion has been suggested.
Let $\calV = (V, \rho, \Boxed{<})$ and $\calW = (W, \sigma, \Boxed{\sqsubset})$ be
two digraphs with linear extensions.
Then each homomorphism $f : (V, \rho) \to (W, \sigma)$ induces a mapping $\widehat f : \rho \to \sigma$ by:
$
  \widehat f(v_1, v_2) = (f(v_1), f(v_2))
$.
A homomorphism $f : (V, \rho) \to (W, \sigma)$ is a \emph{strong rigid quotient map}
from $\calV$ to $\calW$~\cite{masul-DRTRS} if
$\widehat f : (\rho, \Boxed{\lsal}) \to (\sigma, \Boxed{\sqsal})$ is a rigid surjection.
It is rather easy to see that a strong rigid quotient map is a rigid surjection and a quotient map (see~\cite[Lemma~5.2]{masul-DRTRS}).

Let us now present the corresponding notion for ordered oriented graphs.
Let $\calV = (V, \rho, \Boxed{<})$ be an ordered oriented graph. Let
\begin{align*}
  \rho_< &= \Delta_V \union \{(v_1, v_2) \in \rho : v_1 < v_2\}, \text{ and}\\
  \rho_> &= \Delta_V \union \{(v_1, v_2) \in \rho : v_1 > v_2\}.
\end{align*}
Note that both $(V, \rho_<, \Boxed{<})$ and $(V, (\rho_>)^{-1}, \Boxed{<})$ are digraphs with linear extensions.

\begin{DEF}\label{drpog.def.srqm}
  Let $\calV = (V, \rho, \Boxed<)$ and $\calW = (W, \sigma, \Boxed\sqsubset)$ be finite
  ordered oriented graphs and $f : (V, \rho) \to (W, \sigma)$ a homomorphism.
  Then $f$ is a \emph{strong rigid quotient map} between $\calV$ and $\calW$
  if $\widehat f : (\rho_<, \Boxed{\lsal}) \to (\sigma_\sqsubset, \Boxed{\sqsal})$ is a rigid surjection,
  and $\widehat f : ((\rho_>)^{-1}, \Boxed{\lsal}) \to ((\sigma_\sqsupset)^{-1}, \Boxed{\sqsal})$ is a rigid surjection.
\end{DEF}

\begin{LEM}
  A strong rigid quotient map between two finite ordered oriented graphs
  is a rigid surjection and a quotient map.
\end{LEM}
\begin{proof}
Let $\calV = (V, \rho, \Boxed<)$ and $\calW = (W, \sigma, \Boxed\sqsubset)$ be two
finite ordered oriented graphs, and let $f : \calV \to \calW$ be a strong rigid quotient map
between them. What we aim to prove is that $f : (V, \Boxed<) \to (W, \Boxed\sqsubset)$ is a rigid
surjection whereas at the same time $f : (V, \rho) \to (W, \sigma)$ is a quotient map.

We begin the proof by showing that $f$ is indeed surjective. Take any $v \in W$. Then
$(v,v) \in \sigma$ due to the fact that $\sigma$ is reflexive. Therefore, there exists an
$(u_1, u_2) \in \rho$ (regardless of whether $(u_1, u_2)$ belongs to $\rho_<$ or $\rho_>$)
such that $\widehat{f} (u_1, u_2) = (v, v)$ because $\widehat{f}$ is surjective (in both cases).
But then $f(u_1) = v$ which is exactly what we needed, as $u_1 \in V$.

Since $f$ is a homomorphism and both $\widehat{f} : \rho_< \to \sigma_\sqsubset$
and $\widehat{f} : (\rho_>)^{-1} \to (\sigma_\sqsupset)^{-1}$ are surjective it follows immediately
that $f$ must be a quotient map.

Finally, let us prove that $f$ is also a rigid surjection. In other words, taking any
$u, v \in W$ such that $u \sqsubset v$ let us show that $\min f^{-1}(u) < \min f^{-1} (v)$. From
$u \sqsubset v$ it is clear that $(u, u) \sqsal (v, v)$ whence
$\min \widehat{f}^{-1}(u, u) \lsal \min \widehat{f}^{-1} (v, v)$ owing to
$\widehat{f} : \rho_< \to \sigma_\sqsubset$ being
a rigid surjection. Now, let $\min \widehat{f}^{-1}(u, u) = (x_1, x_2)$, where $ x_1 \not= x_2$.
Then $\widehat{f}(x_1, x_2) = (u, u)$ which implies $f(x_1) = u = f(x_2)$. Hence, $\widehat{f}(x_1, x_1) = (u, u)$.
However, it would then appear that $(x_1, x_1) \sqsal (x_1,x_2) = \min \widehat{f}^{-1}(u, u)$ which is
a clear contradiction. So, $\min \widehat{f}^{-1}(u, u)$ must be of the form $(x, x)$ for some (adequate) $x \in V$. We shall
show that $ x = \min f^{-1}(u)$. Assuming the opposite that there exists a $t \in V$, $t \not= x$ such that
$ t = \min f^{-1}(u)$, but $f(x) = u$, we would once again encounter a problem as it would mean that
$t < x \Rightarrow (t, t) \lsal (x, x) = \min \widehat{f}^{-1}(u, u) $, even though $\widehat{f}(t, t) = (f(t), f(t)) = (u, u)$,
which is an obvious contradiction. Analogously, we have every right to denote $\min \widehat{f}^{-1}(v, v)$ with $(y, y)$,
where $y = \min f^{-1}(v)$. At last we see that from $ (x, x) = \min \widehat{f}^{-1}(u, u) \lsal \min \widehat{f}^{-1}(v, v) = (y, y)$
our conclusion $\min f^{-1}(u) = x < y = \min f^{-1}(v)$ follows.
\end{proof}

Because the definition of strong rigid quotient maps for ordered oriented graphs is far from intuitive let us give a
simple example.

\begin{EX}
  Let $\calA = (\{1, 2, 3\}, \alpha, \Boxed<)$ and $\calB = (\{1, 2, 3, 4, 5, 6\}, \beta, \Boxed<)$
  be ordered oriented graphs where the non-loops in $\calA$ are $12$, $23$ and $31$, the non-loops in $\calB$ are $12$, $23$,
  $34$, $45$, $56$ and $61$, and $<$ is the usual ordering of the integers. (In this example only
  we shall write $ij$ instead of $(i, j)$.) Consider the following surjective homomorphisms $\calB \to \calA$:
  $$
    f = \begin{pmatrix} 1 & 2 & 3 & 4 & 5 & 6 \\ 1 & 2 & 3 & 1 & 2 & 3 \end{pmatrix},
    \qquad
    g = \begin{pmatrix} 1 & 2 & 3 & 4 & 5 & 6 \\ 1 & 2 & 2 & 3 & 3 & 3 \end{pmatrix}.
  $$
  Then $f$ is \emph{not} a strong rigid quotient map $\calB \to \calA$ because $\widehat f : \beta_< \to \alpha_<$ is not
  well defined. Namely, $34 \in \beta_<$ but $\widehat f(34) = 31 \notin \alpha_<$. On the other hand, $g$ is a strong
  rigid quotient map $\calB \to \calA$ as both
  $$
    \widehat g : \beta_< \to \alpha_< : \left(\begin{array}{@{}ccccccccccc@{}}
      11 & 22 & 33 & 44 & 55 & 66 & 12 & 23 & 34 & 45 & 56 \\
      11 & 22 & 22 & 33 & 33 & 33 & 12 & 22 & 23 & 33 & 33
    \end{array}\right)
  $$
  and
  $$
    \widehat g : (\beta_>)^{-1} \to (\alpha_>)^{-1} : \left(\begin{array}{@{}ccccccc@{}}
      11 & 22 & 33 & 44 & 55 & 66 & 16 \\
      11 & 22 & 22 & 33 & 33 & 33 & 13
    \end{array}\right)
  $$
  are well-defined rigid surjections.
\end{EX}

\section{Category theory and the Ramsey property}
\label{drpog.sec.rplct}

In order to specify a \emph{category} $\CC$ one has to specify
a class of objects $\Ob(\CC)$, a set of morphisms $\hom_\CC(A, B)$ for all $A, B \in \Ob(\CC)$,
the identity morphism $\id_A$ for all $A \in \Ob(\CC)$, and
the composition of mor\-phi\-sms~$\cdot$~so that
$\id_B \cdot f = f \cdot \id_A = f$ for all $f \in \hom_\CC(A, B)$, and
$(f \cdot g) \cdot h = f \cdot (g \cdot h)$ whenever the compositions are defined.
A morphism $f \in \hom_\CC(B, C)$ is \emph{monic} or \emph{left cancellable} if
$f \cdot g = f \cdot h$ implies $g = h$ for all $g, h \in \hom_\CC(A, B)$ where $A \in \Ob(\CC)$ is arbitrary.
A morphism $f \in \hom_\CC(B, C)$ is \emph{epi} or \emph{right cancellable} if
$g \cdot f = h \cdot f$ implies $g = h$ for all $g, h \in \hom_\CC(C, D)$ where $D \in \Ob(\CC)$ is arbitrary.

\begin{EX}\label{drpog.ex.CH-def}
  Finite chains and embeddings constitute a category that we denote by~$\CH$.
\end{EX}

\begin{EX}\label{drpog.ex.CHrs-def}
  The composition of two rigid surjections is again a rigid surjection, so finite chains and rigid surjections constitute a category
  which we denote by~$\CHrs$.
\end{EX}

\begin{EX}\label{drpog.ex.EDIGsrq}
  Finite digraphs with linear extensions together with strong rigid quotient maps constitute a category
  which we denote by~$\EDIGsrq$.
\end{EX}

\begin{LEM}\label{drpog.lem.EDIGsrq}
  Finite ordered oriented graphs together with strong rigid quotient maps as introduced in Definition~\ref{drpog.def.srqm}
  constitute a category which we denote by~$\OOGRAsrq$.
\end{LEM}
\begin{proof}
  We only have to show that the composition of strong rigid quotient maps of ordered oriented graphs
  (Definition~\ref{drpog.def.srqm}) is again a strong rigid quotient map of ordered oriented graphs.
  
  So, let us consider the composition $g \cdot f$ of two strong rigid quotient maps $f : (V, \rho, \Boxed<) \to (W, \sigma, \Boxed\sqsubset)$
  and $g : (W, \sigma, \Boxed\sqsubset) \to (T, \psi, \Boxed\subset)$. Knowing that both $f$ and $g$ are homomorphisms it follows easily
  that their composition $g \cdot f : (V, \rho) \to (T, \psi)$ is a homomorphism, too. What remains to be shown is that
  $\widehat{g \cdot f} : (\rho_<, \lsal) \to (\psi_\subset, \ssal)$ and
  $\widehat{g \cdot f} : ((\rho_>)^{-1}, \lsal) \to ((\psi_\supset)^{-1}, \ssal)$ are rigid surjections.

  Firstly, notice that $\widehat{g \cdot f} = \widehat{g} \cdot \widehat{f}$ and that
  $\widehat{g \cdot f}^{-1} = \widehat{f}^{-1} \cdot \widehat{g}^{-1}$. Since both $\widehat{f}$ and $\widehat{g}$ are
  surjections it follows that their composition $\widehat{g} \cdot \widehat{f} = \widehat{g \cdot f}$ is surjective.
  Now, take any $(x, y), (z, t) \in \psi_\subset$ such that $(x, y) \ssal (z, t)$.
  Since $\widehat{g} : (\sigma_\sqsubset, \Boxed{\sqsal}) \to (\psi_\subset, \Boxed{\ssal})$ is a rigid surjection,
  we have $ \min \widehat{g}^{-1}(x, y) \sqsal \min \widehat{g}^{-1}(z, t)$.
  Similarly, due to the fact that $\widehat{f} : (\rho_<, \Boxed{\lsal}) \to (\sigma_\sqsubset, \Boxed{\sqsal})$
  is a rigid surjection it follows that
  $$
    \min \widehat{f}^{-1} (\min \widehat{g}^{-1}(x, y)) \lsal \min \widehat{f}^{-1}(\min \widehat{g}^{-1}(z, t)).
  $$
  It is easy to show that for any $S \subseteq \rho_\sqsubset$ we have $\min \widehat{f}^{-1}(\min S) =
  \min \widehat{f}^{-1}(S)$. Therefore,
  \begin{align*}
    \min \widehat{g \cdot f}^{-1} (x, y)
    & = \min \widehat{f}^{-1} ( \widehat{g}^{-1}(x, y) ) = \min \widehat{f}^{-1}(\min \widehat{g}^{-1}(x, y))\\
    &\lsal \min \widehat{f}^{-1}(\min \widehat{g}^{-1}(z, t))\\
    &= \min \widehat{f}^{-1} ( \widehat{g}^{-1}(z, t) ) = \min \widehat{g \cdot f}^{-1} (z, t),
  \end{align*}
  which confirms the claim that
  $\widehat{g \cdot f} : (\rho_<, \lsal) \to (\psi_\subset, \ssal)$ is a rigid surjection.
  By the same argument,
  $\widehat{g \cdot f} : ((\rho_>)^{-1}, \lsal) \to ((\psi_\supset)^{-1}, \ssal)$ is a rigid surjection.
\end{proof}

For a category $\CC$, the \emph{opposite category}, denoted by $\CC^\op$, is the category whose objects
are the objects of $\CC$, morphisms are formally reversed so that
$
  \hom_{\CC^\op}(A, B) = \hom_\CC(B, A)
$,
and so is the composition:
$
  f \cdot_{\CC^\op} g = g \cdot_\CC f
$.

A category $\DD$ is a \emph{subcategory} of a category $\CC$ if $\Ob(\DD) \subseteq \Ob(\CC)$ and
$\hom_\DD(A, B) \subseteq \hom_\CC(A, B)$ for all $A, B \in \Ob(\DD)$.
A category $\DD$ is a \emph{full subcategory} of a category $\CC$ if $\Ob(\DD) \subseteq \Ob(\CC)$ and
$\hom_\DD(A, B) = \hom_\CC(A, B)$ for all $A, B \in \Ob(\DD)$.

A \emph{functor} $F : \CC \to \DD$ from a category $\CC$ to a category $\DD$ maps $\Ob(\CC)$ to
$\Ob(\DD)$ and maps morphisms of $\CC$ to morphisms of $\DD$ so that
$F(f) \in \hom_\DD(F(A), F(B))$ whenever $f \in \hom_\CC(A, B)$, $F(f \cdot g) = F(f) \cdot F(g)$ whenever
$f \cdot g$ is defined, and $F(\id_A) = \id_{F(A)}$.

Categories $\CC$ and $\DD$ are \emph{isomorphic} if there exist functors $F : \CC \to \DD$ and $G : \DD \to \CC$ which are
inverses of one another both on objects and on morphisms.

The \emph{product} of categories $\CC_1$ and $\CC_2$ is the category $\CC_1 \times \CC_2$ whose objects are pairs $(A_1, A_2)$
where $A_1 \in \Ob(\CC_1)$ and $A_2 \in \Ob(\CC_2)$, morphisms are pairs $(f_1, f_2) : (A_1, A_2) \to (B_1, B_2)$ where
$f_1 : A_1 \to B_1$ is a morphism in $\CC_1$ and $f_2 : A_2 \to B_2$ is a morphism in $\CC_2$.
The composition of morphisms is carried out componentwise: $(f_1, f_2) \cdot (g_1, g_2) = (f_1 \cdot g_1, f_2 \cdot g_2)$.

Let $\CC$ be a category and $\calS$ a set. We say that
$
  \calS = \calX_1 \union \ldots \union \calX_k
$
is a \emph{$k$-coloring} of $\calS$ if $\calX_i \cap \calX_j = \0$ whenever $i \ne j$.
Equivalently, a $k$-coloring of $\calS$ is any map $\chi : \calS \to \{1, 2, \ldots, k\}$.
For an integer $k \ge 2$ and $A, B, C \in \Ob(\CC)$ we write
$
  C \longrightarrow (B)^{A}_k
$
to denote that for every $k$-coloring
$
  \hom_\CC(A, C) = \calX_1 \union \ldots \union \calX_k
$
there is an $i \in \{1, \ldots, k\}$ and a morphism $w \in \hom_\CC(B, C)$ such that
$w \cdot \hom_\CC(A, B) \subseteq \calX_i$.

\begin{DEF}
  A category $\CC$ has the \emph{Ramsey property} if
  for every integer $k \ge 2$ and all $A, B \in \Ob(\CC)$
  such that $\hom_\CC(A, B) \ne \0$ there is a
  $C \in \Ob(\CC)$ such that $C \longrightarrow (B)^{A}_k$.

  A category $\CC$ has the \emph{dual Ramsey property} if $\CC^\op$ has the Ramsey property.
\end{DEF}

Clearly, if $\CC$ and $\DD$ are isomorphic categories and one of them has the (dual) Ramsey property,
then so does the other. Actually, even more is true: if $\CC$ and $\DD$ are equivalent categories
and one of them has the (dual) Ramsey property, then so does the other. We refrain from providing the
definition of (the fairly standard notion of) categorical equivalence as we shall have no use for it
in this paper, and for the proof we refer the reader to~\cite{masulovic-ramsey}.

\begin{EX}\label{cerp.ex.FRT-ch}
  The category $\CH$ (see Example~\ref{drpog.ex.CH-def})
  has the Ramsey property. This is just a reformulation of the Finite Ramsey Theorem (Theorem~\ref{dthms.thm.frt}).
\end{EX}

\begin{EX}\label{cerp.ex.FDRT-ch}
  The category $\CHrs$ (see Example~\ref{drpog.ex.CHrs-def})
  has the dual Ramsey property. This is just a reformulation of the Finite Dual Ramsey Theorem (Theorem~\ref{drpog.thm.FDRT};
  see also the discussion in the Introduction.)
\end{EX}

\begin{EX}\label{drpog.ex.EDIGsrq-RP}
  The category $\EDIGsrq$ (see Example~\ref{drpog.ex.EDIGsrq})
  has the dual Ramsey property~\cite{masul-DRTRS}.
\end{EX}

\section{The main result}
\label{drpog.sec.main}

Our goal in this paper is to prove that the category $\OOGRAsrq$ has the dual Ramsey property.
In order to do so, we shall employ a strategy devised in~\cite{masul-drp-perm}. Let us recall two technical statements
from~\cite{masul-drp-perm}.

A \emph{diagram} in a category $\CC$ is a functor $F : \Delta \to \CC$ where the category $\Delta$ is referred to as the
\emph{shape of the diagram}. Given a diagram $F : \Delta \to \CC$,
an object $C \in \Ob(\CC)$ and a family of morphisms $(h_\delta : F(\delta) \to C)_{\delta \in \Ob(\Delta)}$
form a \emph{commuting cocone in $\CC$ over~$F$}
if $h_\gamma \cdot F(g) = h_\delta$ for every morphism $g : \delta \to \gamma$ in $\Delta$:
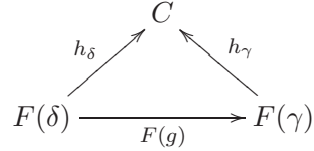
$$
  \xymatrix{
     & C & \\
    F(\delta) \ar[ur]^{h_\delta} \ar[rr]_{F(g)} & & F(\gamma) \ar[ul]_{h_\gamma}
  }
$$
We then say that the diagram $F$ \emph{has a commuting cocone in $\CC$}.

A \emph{binary category} is a finite, acyclic, bipartite digraph with loops
where all the arrows go from one class of vertices into the other
and the out-degree of all the vertices in the first class is~2 (modulo loops):
$$
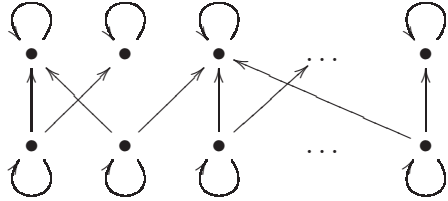
\xymatrix{
  & \\
  \bullet \ar@(ur,ul) & \bullet \ar@(ur,ul) & \bullet \ar@(ur,ul) & \ldots & \bullet \ar@(ur,ul) \\
  \bullet \ar@(dr,dl) \ar[u] \ar[ur] & \bullet \ar@(dr,dl) \ar[ur] \ar[ul] & \bullet \ar@(dr,dl) \ar[u] \ar[ur] & \ldots & \bullet \ar@(dr,dl) \ar[u] \ar[ull] \\
  &
}
$$
A \emph{binary diagram} in a category $\CC$ is a functor $F : \Delta \to \CC$ where $\Delta$ is a binary category,
$F$ takes the bottom row of $\Delta$ onto the same object, and takes the top row of $\Delta$ onto
the same object, Fig.~\ref{nrt.fig.2}.
A subcategory $\DD$ of a category $\CC$ is \emph{closed for binary diagrams} if every binary diagram
$F : \Delta \to \DD$ which has a commuting cocone in $\CC$ has a commuting cocone in~$\DD$.

\begin{figure}
  $$
  \xymatrix{
    \bullet \ar@(ur,ul) & \bullet \ar@(ur,ul) & \bullet \ar@(ur,ul)
    & & B & B & B
  \\
    \bullet \ar@(dr,dl) \ar[u] \ar[ur] & \bullet \ar@(dr,dl) \ar[ur] \ar[ul] & \bullet \ar@(dr,dl) \ar[ul] \ar[u]
    & & A \ar[u]^{f_1} \ar[ur]_(0.3){f_2} & A \ar[ur]^(0.3){f_4} \ar[ul]_(0.3){f_3} & A \ar[ul]^(0.3){f_5} \ar[u]_{f_6}
  \\
    & \Delta \ar[rrrr]^F  & & & & \CC
  }
  $$
  \caption{A binary diagram in $\CC$ (of shape $\Delta$)}
  \label{nrt.fig.2}
\end{figure}
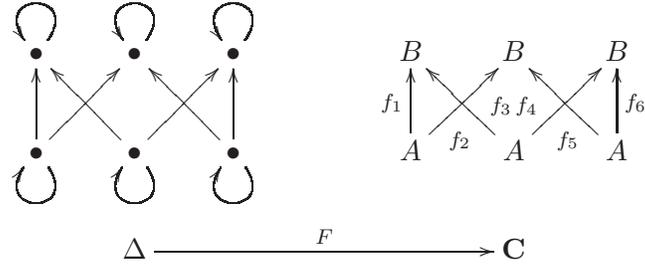

\begin{THM}\label{nrt.thm.1} \cite{masul-drp-perm}
  Let $\CC$ be a category such that every morphism in $\CC$ is monic and
  such that $\hom_\CC(A, B)$ is finite for all $A, B \in \Ob(\CC)$, and let $\DD$ be a
  (not necessarily full) subcategory of~$\CC$.
  If $\CC$ has the Ramsey property and $\DD$ is closed for binary diagrams, then $\DD$ has the Ramsey property.
\end{THM}

We shall also need a categorical version of the Product Ramsey Theorem for Finite Structures of M.~Soki\'c~\cite{sokic2}.
We proved this statement in the categorical context in~\cite{masul-drp-perm} where we used this abstract version
to prove that the class of finite permutations has the dual Ramsey property.

\begin{THM}\label{sokic-prod} \cite{masul-drp-perm}
  Let $\CC_1$ and $\CC_2$ be categories such that $\hom_{\CC_i}(A, B)$ is finite
  for all $A, B \in \Ob(\CC_i)$, $i \in \{1, 2\}$.
  If $\CC_1$ and $\CC_2$ both have the Ramsey property then
  $\CC_1 \times \CC_2$ has the Ramsey property.
\end{THM}

The following dual Ramsey theorem for finite ordered oriented graphs is the main result of the paper.

\begin{THM}
  The category $\OOGRAsrq$ has the dual Ramsey property.
\end{THM}
\begin{proof}
  The category $\CC^\op = \EDIGsrq^\op \times \EDIGsrq^\op$ has the Ramsey property
  (Example~\ref{drpog.ex.EDIGsrq-RP} and Theorem~\ref{sokic-prod}).
  Let $\DD$ be the following subcategory of $\CC$. For each $(V, \rho, \Boxed<) \in \Ob(\OOGRAsrq)$
  the pair $\big((V, \rho_< , \Boxed<), (V, (\rho_>)^{-1}, \Boxed<)\big)$ is an object in $\DD$
  and these are the only objects in~$\DD$.
  Morphisms of $\DD$ are pairs of morphisms from $\EDIGsrq$ of the form
  $$
    (f, f) : \big((V, \rho_<,   \Boxed<), (V, (\rho_>)^{-1},   \Boxed<)\big) \to
             \big((W, \sigma_\sqsubset, \Boxed\sqsubset), (W, (\sigma_\sqsupset)^{-1}, \Boxed\sqsubset)\big).
  $$
  Clearly, a pair $\big((V, \rho, \Boxed<), (W, \sigma, \Boxed\sqsubset)\big) \in \Ob(\CC)$
  belongs to $\Ob(\DD)$ if and only if $V = W$, $\Boxed< = \Boxed\sqsubset$,
  $\rho \cap \sigma = \Delta_V$ and $\rho \cap \sigma^{-1} = \Delta_V$.

  Following Theorem~\ref{nrt.thm.1} it suffices to show that $\DD^\op$ is a subcategory of $\CC^\op$ closed for binary diagrams.
  Take any $\calA = ((A, \alpha_{<} , \Boxed<), (A, (\alpha_{>})^{-1}, \Boxed<))$ and
  $\calB = ((B, \beta_{\prec} , \Boxed\prec), (B, (\beta_{\succ})^{-1}, \Boxed\prec)) \in \Ob(\DD)$
  and let $F : \Delta \to \DD^\op$ be a binary diagram which takes the top row of $\Delta$ to $\calB$, takes the bottom row
  of $\Delta$ to $\calA$ and which has a commuting cocone in $\CC^\op$. Let
  $((V, \rho, \Boxed\ltimes), (W, \sigma, \Boxed\vartriangleleft))$ together with the morphisms
  $e_i = (f_i, g_i)$, $1 \le i \le k$, be a commuting cocone in $\CC^\op$ over~$F$:
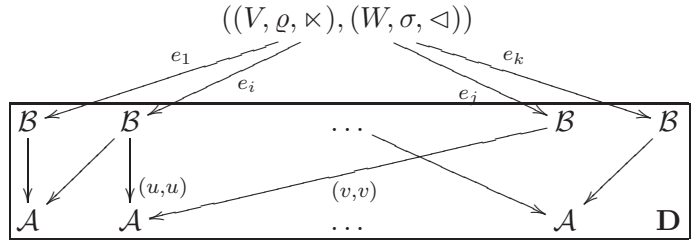
  $$
  \xymatrix{
    & & ((V, \rho, \Boxed\ltimes), (W, \sigma, \Boxed\vartriangleleft)) \ar[dll]_{e_1} \ar[dl]^{e_i} \ar[dr]_(0.6){e_j} \ar[drr]^{e_k}
  \\
    \calB \ar[d]  & \calB \ar[dl] \ar[d]^(0.65){(u,u)}  & \ldots \ar[dr] & \calB \ar[dll]^(0.5){(v,v)} & \calB \ar[dl]
  \\
    \calA   & \calA  & \ldots & \calA   & \DD
    \save "2,1"."3,5"*[F]\frm{} \restore
  }
  $$
  (Note that the arrows in the diagram are reversed because the diagram depicts a situation in $\CC$.)
  Without loss of generality we may assume that $V \cap W = \0$.
  Recall that $\widehat{f_i} : (\rho, \Boxed{\ltimes_{sal}}) \to (\beta_{\prec}, \Boxed{\prec_{sal}})$ and
  $\widehat{g_i} : (\sigma, \Boxed{\vartriangleleft_{sal}}) \to ((\beta_{\succ})^{-1}, \Boxed{\prec_{sal}})$
  are rigid surjections.

  Now, let $D = V \union W$, $\delta = \rho \union \sigma^{-1}$ and $\Boxed\sqsubset = \Boxed\ltimes \oplus \Boxed\vartriangleleft$.
  For each $i \in \{1, \ldots, k\}$ define $\phi_i : D \to B$ as follows:
  $$
    \phi_i(x) = \begin{cases}
      f_i(x), & x \in V,\\
      g_i(x), & x \in W.
    \end{cases}
  $$
  The next step would be to prove that $\calD = ((D, \delta_{\sqsubset}, \Boxed\sqsubset),(D, (\delta_{\sqsupset})^{-1}, \Boxed\sqsubset))$
  belongs to $\Ob(\DD)$ and that $(\phi_i, \phi_i) \in \hom_\DD(\calD, \calB)$, for all~$i$.
  The former is easily verifiable bearing in mind that
  $\delta_{\sqsubset} = \rho \union \Delta_D$ and $(\delta_{\sqsupset})^{-1} = \sigma \union \Delta_D$.
  To prove the latter note that $\phi_i : (D, \delta) \to (B, \beta)$, $\phi_i : (D, \delta_\sqsubset) \to (B, \beta_\prec)$
  and $\phi_i : (D, (\delta_\sqsupset)^{-1}) \to (B, (\beta_\succ)^{-1})$ are homomorphisms, so
  we need to show that both $\widehat{\phi_i} : (\delta_{\sqsubset}, \Boxed{\sqsubset_{sal}}) \to
  (\beta_{\prec}, \Boxed{\prec_{sal}})$ and
  $\widehat{\phi_i} : ((\delta_{\sqsupset})^{-1}, \Boxed{\sqsubset_{sal}}) \to ((\beta_{\succ})^{-1}, \Boxed{\prec_{sal}})$
  are rigid surjections. As the proof of rigidity in the second case does not differ
  substantially from the one in the first case, aside from some minor technical details, we shall
  present only the first one.

  Basically, what we need to show is that whenever $a \mathrel{\prec_{sal}} b$ for some $a = (a_1, a_2), b = (b_1, b_2) \in \beta_{\prec}$, it
  immediately follows that $x \mathrel{\sqsubset_{sal}} y$, where $x = \min \widehat{\phi_i}^{-1}(a)$ and
  $y = \min \widehat{\phi_i}^{-1}(b)$. Clearly, $x = (x_1,x_2)$ and $y = (y_1, y_2)$
  belong to $\delta_{\sqsubset} = \rho \union \Delta_D = \rho \union \Delta_W$.

  There are four cases to consider, but
  before we begin, let us first notice a trivial, yet useful, fact that since $\Boxed\ltimes \subseteq \Boxed\sqsubset$
  and $\Boxed\vartriangleleft \subseteq \Boxed\sqsubset$ we have that $\Boxed{\ltimes_{sal}} \subseteq \Boxed{\sqsubset_{sal}}$
  and $\Boxed{\vartriangleleft_{sal}} \subseteq \Boxed{\sqsubset_{sal}}$.

  Case 1: $x, y \in \rho \subseteq V^2$. Then $a = \widehat{\phi_i}(x) = \widehat{f_i}(x)$ and
  $b = \widehat{\phi_i}(y) = \widehat{f_i}(y)$, bearing in mind (at all times)
  the very definition of $\phi_i$. Consequently, $a \mathrel{\prec_{sal}} b$ implies
  $x = \min \widehat{\phi_i}^{-1}(a) = \min \widehat{f_i}^{-1}(a) \mathrel{\ltimes_{sal}}
  \min \widehat{f_i}^{-1}(b) = \min \widehat{\phi_i}^{-1}(b) = y$, since $\widehat{f_i}$ is
  a rigid surjection. Finally, $x \mathrel{\sqsubset_{sal}} y$ because $\Boxed{\ltimes_{sal}} \subseteq \Boxed{\sqsubset_{sal}}$.

  Case 2: $x, y \in \Delta_W \subseteq W^2$. Similarly as in Case~1 we have that
  $a \mathrel{\prec_{sal}} b$ implies $x = \min \widehat{\phi_i}^{-1}(a) = \min \widehat{g_i}^{-1}(a) \mathrel{\vartriangleleft_{sal}}
  \min \widehat{g_i}^{-1}(b) = \min \widehat{\phi_i}^{-1}(b) = y$, only this time it is the fact that
  $\widehat{g_i}$ is a rigid surjection which yields $x \mathrel{\sqsubset_{sal}} y$.

  Case 3: $x \in \rho \subseteq V^2$ and $y \in \Delta_W \subseteq W^2$.
  Note, first, that $b_1 = b_2$ since $b = \widehat{\phi_i}(y) = (\phi_i(y_1), \phi_i(y_2))$, and $y_1 = y_2$ due to
  $y \in \Delta_W$. The assumption $a \mathrel{\prec_{sal}} b$ now implies that $a_1 = a_2$ as well, whence
  $x \in \Delta_V$. The fact that $\Boxed\sqsubset = \Boxed\ltimes \oplus \Boxed\vartriangleleft$
  immediately leads to~$x \mathrel{\sqsubset_{sal}} y$.

  Case 4: $x \in \Delta_W \subseteq W^2$ and $y \in \rho \subseteq V^2$. Similarly as in Case~3 we have that
  $a_1 = a_2$. If $b_1 \ne b_2$ then as a consequence we have that $y_1 \ne y_2$, so
  $x \mathrel{\sqsubset_{sal}} y$ by definition of $\sqsubset_{sal}$.

  Assume, therefore, that $b_1 = b_2$. Let us show that then $y_1 = y_2$. Assume this is not the case. Then
  $\widehat{\phi_i}(y_1, y_1) = (b_1, b_1) = (b_1, b_2) = \widehat{\phi_i}(y_1, y_2)$. Therefore,
  $(y_1, y_1) \in \widehat{\phi_i}^{-1}(b)$. On the other hand, we know that
  $y = \min \widehat{\phi_i}^{-1}(b)$, whence $(y_1, y_2) = y \mathrel{\sqsubseteq_{sal}} (y_1, y_1)$, which
  is impossible by definition of $\sqsubseteq_{sal}$.

  So, we have shown that $y_1 = y_2$, whence $y \in \Delta_V$ and thus $y \mathrel{\sqsubset_{sal}} x$.
  Since $y = \min \widehat{\phi_i}^{-1}(b) \in \Delta_V \subseteq V^2 $ it follows that
  $\min \widehat{\phi_i}^{-1}(b) = \min \widehat{f_i}^{-1}(b)$. Knowing that $\widehat{f_i}$ is a rigid
  surjection (or in other words that it maps an initial segment of a chain onto an initial segment of the other chain)
  we may now claim the existence of $z \in \Delta_V$ such that $z \mathrel{\sqsubset_{sal}} y$
  for which $\widehat{f_i}(z) = a$. Clearly $\widehat{\phi_i}(z) = a$, but bearing in mind that
  $z \mathrel{\sqsubset_{sal}} y \mathrel{\sqsubset_{sal}} x$
  we come to a contradiction with the fact that $x = \min \widehat{\phi_i}^{-1}(a)$.

  Finally, what remains to be checked is that $ (u, u) \cdot (\phi_i, \phi_i) = (v, v) \cdot (\phi_j, \phi_j) $ whenever
  $(u, u) \cdot  e_i  = (v, v) \cdot e_j  $.
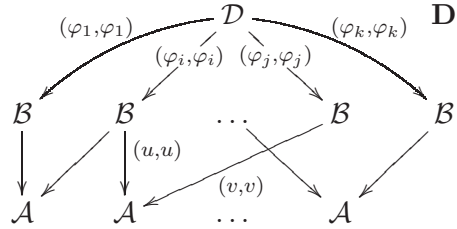
  $$
  \xymatrix{
    & & \calD \ar@/_3mm/[dll]_{(\phi_1, \phi_1)}  \ar[dl]|(0.4){(\phi_i, \phi_i)} \ar[dr]|(0.4){(\phi_j, \phi_j)} \ar@/^3mm/[drr]^{(\phi_k, \phi_k)}& & \DD
  \\
    \calB \ar[d] & \calB \ar[dl] \ar[d]^(0.35){(u,u)} & \ldots \ar[dr] & \calB \ar[dll]^(0.525){(v,v)} & \calB \ar[dl]
  \\
    \calA   & \calA   & \ldots & \calA
  }
  $$
  Assume that  $(u, u) \cdot  e_i  = (v, v) \cdot e_j  $. Then
  $u \cdot f_i = v \cdot f_j$ and $u \cdot g_i = v \cdot g_j$.  Now, take any $x \in D$.
  If $x \in V$ then $u \cdot \phi_i(x) = u (\phi_i(x)) = u (f_i(x)) = (u \cdot f_i) (x) =  (v \cdot f_j) (x) =
                     v (f_j(x)) = v (\phi_j(x)) =  v \cdot \phi_j(x)$.
  If, on the other hand, $x \in W$ then $u \cdot \phi_i(x) = u (\phi_i(x)) = u (g_i(x)) = (u \cdot g_i) (x) =  (v \cdot g_j) (x) =
                     v (g_j(x)) = v (\phi_j(x)) =  v \cdot \phi_j(x)$.
  This concludes the proof.
\end{proof}

\section{Acknowledgements}

The authors gratefully acknowledges the support of the Grant No.\ 174019 of the Ministry of Education,
Science and Technological Development of the Republic of Serbia.

\end{document}